\documentclass{amsart}
\usepackage{hyperref}
\usepackage{graphicx}
\usepackage{amscd}
\usepackage{amsmath}
\usepackage{amssymb}
\usepackage{verbatim}
\newtheorem{conjecture}{Conjecture}
\newtheorem{question}{Question}
\usepackage[dvipsnames]{xcolor}
\usepackage{listings}
\definecolor{dkgreen}{rgb}{0,0.6,0}
\definecolor{gray}{rgb}{0.5,0.5,0.5}
\definecolor{mauve}{rgb}{0.58,0,0.82}
\usepackage{pgfplots}
\pgfplotsset{compat=newest}
\usetikzlibrary{arrows.meta}
\usetikzlibrary{bending}
\usepackage{bm}
\usepackage{asymptote}
\newtheorem{theorem}{Theorem}[section]
\newtheorem{lemma}[theorem]{Lemma}

\theoremstyle{definition}
\newtheorem{definition}[theorem]{Definition}
\newtheorem{proposition}[theorem]{Proposition}

\newtheorem{corollary}[theorem]{Corollary}

\theoremstyle{remark}

\numberwithin{equation}{section}

\newcommand\isomto{\stackrel{\sim}{\smash{\longrightarrow}\rule{0pt}{0.4ex}}}
\begin{document}

\title{Notes on the Hodge Conjecture for Fermat Varieties}

\author{Genival Da Silva Jr.}



\begin{abstract}
We review a combinatoric approach to the Hodge Conjecture for Fermat Varieties and announce new cases where the conjecture is true.
\end{abstract}

\maketitle

\section{Introduction}
The Hodge conjecture is major open problem in Complex Algebraic Geometry that has been puzzling mathematician for decades now. The modern statement is the following: Let $X$ be smooth complex projective variety, then the (rational) cycle class map is surjective:$$cl_{\otimes\mathbb{Q}}:CH^p(X)\otimes\mathbb{Q}\rightarrow H^{p,p}\cap H^{2p}(X,\mathbb{Q})$$
where $cl_{\otimes\mathbb{Q}}(\sum a_iX_i) = \sum a_i[X_i]$, $a_i\in\mathbb{Q}$ and $[X_i]$ is the class of the subvariety $X_i$.

The case $p=1$ is the only case that it is known to hold in general, which follows from Lefschetz's theorem on $(1,1)$-classes. Special cases have emerged during the years but all of them were specific for certain classes of varieties. For example, Abelian varieties of prime dimension, unirational and uniruled fourfolds, hypersurfaces of degree less than 6, and some others \cite{L}.

Using hard Lefschetz theorem, Lefschetz hyperplane theorem and some Hilbert scheme arguments, we can reduce the Hodge conjecture to the case of an even dimensional $(>2)$ variety and primitive middle cohomology classes.

Shioda \cite{S} gave an interesting characterization of the Hodge conjecture for Fermat varieties, which we now review.

\section{Shioda's work}
Let $X^n_m\in\mathbb{P}^{n+1}$ denote the Fermat variety of dimension $n$ and degree $m$, i.e. the solution to the equation:
$$x_0^m+x_1^m+\ldots+x_{n+1}^m=0$$
and $\mu_m$ the group of $m$-th roots of unity. Let $G^n_m$ be quotient of the group              
$\overbrace{\mu_m\times\ldots\times\mu_m}^{n+2}$ by the subgroup of diagonal elements.

The group $G^n_m$ acts naturally on $X^n_m$ by coordinatewise multiplication, moreover, the character group $\hat{G^n_m}$ of $G^n_m$ can be identified with the group: 
$$\hat{G^n_m}=\{ (a_0,\ldots, a_{n+1})|a_i\in\mathbb{Z}_m, a_0+\ldots+a_{n+1}=0\}$$ via $(\zeta_0,\ldots,\zeta_{n+1})\mapsto \zeta_0^{a_0}\dots\zeta_{n+1}^{a_{n+1}}$, where $(\zeta_0,\ldots,\zeta_{n+1})\in G^n_m$.

By the previous section, in order to prove the Hodge conjecture, it's enough to prove it for primitive classes, therefore in this paper we will focus on primitive cohomology. The action of $G^n_m$ extends to the primitive cohomology and makes $H^i_{prim}(X^n_m,\mathbb{Q})$ and $H^i_{prim}(X^n_m,\mathbb{C})$ a $G^n_m$-module. For $\alpha\in \hat{G^n_m}$, we set:
$$V(\alpha)=\{ \xi\in H^n_{prim}(X^n_m,\mathbb{C}) | g^*(\xi)=\alpha(g)\xi\text{ for all } g\in G^n_m \}$$

Before stating the characterization of Hodge classes we need a few notation. Let
$$\mathfrak{U}^n_m:= \{ \alpha = (a_0,\ldots, a_{n+1})\in \hat{G^n_m} | a_i\neq 0 \text{ for all } i \}$$
For $\alpha\in\mathfrak{U}^n_m$ we set $|\alpha|=\sum_i \frac{<a_i>}{m}$, where $<a_i>$ is the representative of $a_i\in\mathbb{Z}_m$ between $1$ and $m-1$. Suppose $n=2p$, then we set
$$\mathfrak{B}^n_m:= \{ \alpha \in \mathfrak{U}^n_m | |t\alpha|=p+1 \text{ for all } t\in\mathbb{Z}^*_m \}$$

\begin{theorem}\cite{R,S} Let $Hdg^p :=  H^{p,p}\cap H^{2p}_{prim}(X,\mathbb{Q})$ be the group of primitive Hodge cycles. Then:
    \begin{itemize}
        \item[(a)] $\dim V(\alpha)=0$ or $1$, and $V(\alpha)\neq 0 \iff \alpha\in\mathfrak{U}^n_m$
        \item[(b)] $Hdg^p = \bigoplus_{\alpha\in\mathfrak{B}^n_m} V(\alpha)$
    \end{itemize}
\end{theorem}

Now let $C(X^n_m)$ denote the subspace of $Hdg^p$ which are classes of algebraic cycles. Then $C(X^n_m)$ is a $G^n_m$-submodule and by the theorem above there is a subset $\mathfrak{C}^n_m\subseteq\mathfrak{B}^n_m$ such that:
$$C(X^n_m)= \bigoplus_{\alpha\in\mathfrak{C}^n_m} V(\alpha)$$
the Hodge conjecture can then be stated as follows:
\begin{conjecture}[Hodge Conjecture]
For all $n,m$ we have $\mathfrak{C}^n_m=\mathfrak{B}^n_m$.
\end{conjecture}
By the discussion in the previous section, this is true for $n\leq2$ and all $m$. The idea to prove this equality for Fermat varieties it to use the fact that $X^n_m$ `contains' disjoint unions of $X^k_m$ with $k<n$, we then blow that up to find a relation between the cohomologies and to inductively construct algebraic cycles in $X^n_m$. More precisely, we have:

\begin{theorem}\cite{S}\label{shioda}
    Let $n=r+s$ with $r,s\geq 1$. Then there is an isomorphism 
    $$f: [H^r_{prim}(X^r_m,\mathbb{C})\otimes H^s_{prim}(X^s_m,\mathbb{C})]^{\mu_m} \oplus H^{r-1}_{prim}(X^{r-1}_m,\mathbb{C})\otimes H^{s-1}_{prim}(X^{s-1}_m,\mathbb{C})\isomto H^n_{prim}(X^n_m,\mathbb{C})$$
    with the following properties:
    \begin{itemize}
        \item[a)] $f$ is $G^n_m$-equivariant
        \item[b)] $f$ is morphism of Hodge structures of type (0,0) on the first summand and of type $(1,1)$ on the second.
        \item[c)]If $n=2p$ then $f$ preserves algebraic cycles, moreover if $$Z_1\otimes Z_2\in H^{r-1}_{prim}(X^{r-1}_m,\mathbb{C})\otimes H^{s-1}_{prim}(X^{s-1}_m,\mathbb{C})$$
        then $f(Z_1\otimes Z_2)=mZ_1\wedge Z_2$, where $Z_1\wedge Z_2$ is the algebraic cycle obtained by joining $Z_1$ and $Z_2$ by lines on $X^n_m$, when $Z_1,Z_2$ are viewed as cycles in $X^n_m$.
    \end{itemize}
\end{theorem}

In light of this theorem, we introduce the following notation:
$$\mathfrak{U}^{r,s}_m=\{(\beta,\gamma)\in\mathfrak{U}^{r}_m\times\mathfrak{U}^{s}_m\ | \beta=(b_0,\ldots,b_{r+1}),\gamma=(c_0,\ldots,c_{s+1}),\text{ and } b_{r+1}+c_{s+1}=0\}$$

For $(\beta,\gamma)\in\mathfrak{U}^{r,s}_m$ we define:
$$\beta\#\gamma=(b_0,\ldots,b_{r},c_0,\ldots,c_{s})\in\mathfrak{U}^{r+s}_m$$
and for $\beta'=(b_0,\ldots,b_{r})\in\mathfrak{U}^{r-1}_m$ and $\gamma'=(c_0,\ldots,c_{s})\in\mathfrak{U}^{s-1}_m$, we set:
$$\beta'*\gamma'=(b_0,\ldots,b_{r},c_0,\ldots,c_{s})\in\mathfrak{U}^{r+s}_m$$

Using the theorem above we have:
\begin{corollary}
Suppose $n=2p=r+s$, where $r,s\geq 1$.
\begin{itemize}
    \item[a)] If $r,s$ are odd and $(\beta',\gamma')\in\mathfrak{C}^{r-1}_m\times \mathfrak{C}^{s-1}_m$ then $\beta'*\gamma'\in\mathfrak{C}^{n}_m$
    \item[b)] If $r,s$ are even and $(\beta,\gamma)\in(\mathfrak{C}^{r}_m\times \mathfrak{C}^{s}_m)\cap \mathfrak{U}^{r,s}_m$ then $\beta\#\gamma\in\mathfrak{C}^{n}_m$
\end{itemize}
\end{corollary}

By the above corollary, the Hodge conjecture can be proven for the Fermat $X^n_m$ if the following conditions are true for every $\alpha\in\mathfrak{B}^{n}_m$:
\begin{itemize}
    \item[(P1)] $\alpha\sim\beta'*\gamma'$ for some $(\beta',\gamma')\in\mathfrak{B}^{r-1}_m\times \mathfrak{B}^{s-1}_m,\quad (r,s \text{ odd})$.
    \item[(P2)] $\alpha\sim\beta\#\gamma$ for some $(\beta,\gamma)\in(\mathfrak{B}^{r}_m\times \mathfrak{B}^{s}_m)\cap \mathfrak{U}^{r,s}_m,\quad (r,s \text{ even and positive})$.
\end{itemize}
where $\sim$ means equality up to permutation between factors.

In order to make these conditions more explicit, we introduce the additive semi-group $M_m$ of non-negative solutions $(x_1,\ldots,x_{m-1};y)$ with $y>0$, to the following system of linear equations:
$$\sum_{i=1}^{m-1} <ti>x_i = my\text{ for all } t\in\mathbb{Z}_m^*$$
Also, define $M_m(y)$ as those solutions where $y$ is fixed. Note that by Gordan's lemma, $M_m$ is finitely generated.
\begin{definition}
An element $a\in M_m$ is called \textbf{decomposable} if $a=c+d$ for some $c,d\in M_m$, otherwise it's called \textbf{indecomposable}. An element is called \textbf{quasi-decomposable} if $a+b=c+d$ for some $a\in M_m(1)$ and $c,d\in M_m$ with $c,d\neq a$.
\end{definition}

With this notation we can identify elements of $\mathfrak{B}^n_m$ with elements of $M_m$ using the map:
$$\{\}:\alpha=(a_0,\ldots,a_{n+1})\in \mathfrak{B}^n_m \mapsto \{\alpha\}=(x_1(\alpha),\ldots,x_{m-1}(\alpha),\frac{n}{2}+1)\in M_m(\frac{n}{2}+1)$$
where $x_k(\alpha)$ is the number os i's such that $<a_i>=k$.

Note that $\alpha$ satisfies $(P1)$ above if and only if $\{\alpha\}$ is decomposable. If $\alpha$ satisfies $(P2)$ then $\{\alpha\}$ is quasi-decomposable. Conversely, if the latter is true then $\alpha$ satisfies $(P1)$ or $(P2)$. So it makes sense to introduce the following conditions:

\begin{itemize}
    \item[($P^n_m$)] Every indecomposable elemets of $M_m(y)$ with $3\leq y\leq \frac{n}{2}+1$, if any, is quasi-decomposable.
    \item[($P_m$)] Every indecomposable elemets of $M_m(y)$ with $y\geq 3$ is quasi-decomposable.
\end{itemize}

By the results above we conclude:
\begin{theorem}\cite{S}
    If condition ($P_m$) is satisfied, then the Hodge conjecture is true for $X^n_m$ for any $n$. If ($P^n_m$) is satisfied then the Hodge conjecture is true for $X^n_m$.
\end{theorem}

For $m$ prime or $m=4$, $M_m$ is generated by $M_m(1)$ which gives:
\begin{theorem}\cite{R,S}
    If $m$ is prime or $m=4$, the Hodge conjecture is true for $X^n_m$ for all $n$.
\end{theorem}
Shioda manually verified condition ($P_m$) for $m\leq 20$ and concluded:

\begin{theorem}\cite{S}
    If $m\leq 20$, the Hodge conjecture is true for $X^n_m$ for all $n$.
\end{theorem}
Starting at $m=21$ the number of indecomposables and the length of elements of $M_m$ are very large  so it's hard to verify ($P_m$) by hand for unknown cases, unless $m=p^2$ a square of a prime. In the latter case, condition ($P_m$) is not always true, it's false for $m=25$ for example. However, Aoki\cite{A} explicitly constructed the algebraic cycles that generates each $V(\alpha)$, such cycles are called \textbf{standard} cycles.
\begin{theorem}\cite{A}
If $m=p^2$,  the Hodge conjecture is true for $X^n_m$ for all $n$, even though condition ($P_m$) may be false.
\end{theorem}

\section{New cases of the Hodge conjecture}
A natural question is whether or not the Hodge conjecture can always be proved using condition ($P_m$). As described above, there are false negatives, i.e. ($P_m$) is false but the Hodge conjecture is still true.

This is due to the fact that there are cycles not coming from the induced structure, see \cite{A}. The next obvious question is then for which values of $m$, if any, the condition ($P_m$) is false, besides $m=p^2$. This would say that there are cycles, not of type standard as in \cite{A}, such that they too do not come from the induced structure. Technically they are candidates for a counter-example to the Hodge conjecture.

We used SAGE math to answer that question and more generally to investigate when conditin ($P_m$) is true. All the code used in this section can be found in the \hyperref[appendix]{Appendix}.

In the case of Fermat fourfolds, we computed first all the indecomposable elements with length $\geq 3$, because the $(2,2)$ cycles have length exactly $3$. So the idea was to find values of $m$ for which there were none of them.

\begin{proposition}
If $m\leq 100$ is an integer coprime to 6, then the Hodge conjecture is true for all Fermat fourfolds $X^4_m$.
\end{proposition}

This is a strong evidence that the Hodge conjecture should be true for fourfolds $X^4_m$ where $m$ is coprime to $6$. It also suggests the structure of $\mathfrak{B}^4_m$, namely, if $3 \mid m$ then there are indecomposables elements of length $3$.

That is indeed the case, before proving it we need the lemma below proved by Shioda and Aoki:
\begin{lemma}Let $m$ be an integer coprime to 6.
\begin{itemize}
    \item[a)] (Shioda \cite{S2}) $Hdg^1(X^2_m)$ is generated by lines.
    \item[b)] (Aoki \cite{aoki}) $Hdg^1(X^1_m\times X^1_m)$ is generated by lines.
\end{itemize}
\end{lemma}

\begin{theorem}
If $m$ is coprime to $6$, then the Hodge conjecture is true for all Fermat fourfolds $X^4_m$.
\end{theorem}
\begin{proof}
By taking $m=6$ and $r=s=2$ in theorem \ref{shioda}, the result follows directly from the lemma above.
\end{proof}
The following corollary is immediate by \cite{S}:
\begin{corollary}
If $m_i$ are integers coprime to $6$, then the Hodge conjecture is true for arbitrary products of Fermat fourfolds $X^4_{m_1}\times \dots X^4_{m_k}$.
\end{corollary}

We slightly extended Shioda's work by verifying condition ($P_m$) for $m=21,27$. A computational proof can be found in the \hyperref[appendix]{Appendix}.
\begin{theorem}
The Hodge conjecture is true for Fermats $X^n_{21}$ and $X^n_{27}$.
\end{theorem}

An interesting case is $m=33$, where condition ($P_m$) is false, because we've explicitly founded a cycle that is not quasi-decomposable and is not of standard type either.

\begin{proposition}
Condition $P_{33}$ is false. More precisely, the following cycle 
$$(0, 0, 0, 0, 0, 0, 1, 0, 0, 1, 0, 0, 1, 0, 0, 0, 0, 0, 1, 0, 0, 1, 0, 0, 0, 0, 0, 1, 0, 0, 0, 0, 3)$$
supported on $X^4_{33}$ is not quasi-decomposable in $M_{33}$.
\end{proposition}

This proposition confirms that starting at $n=4$, there are cycles not coming from the induced structure. Therefore, we can not prove the Hodge conjecture only using this approach. One thing that can be done is to find explicitly the algebraic cycles whose class project non trivially to $V(\alpha)$ for each $\alpha\in\mathfrak{B}^n_m$, see \cite{A}.

In the particular case where $m=3d$ and $3\nmid d$, as above, we have a candidate. Consider the following elementary symmetric polynomials in \textbf{x}$=(x_0,\dots,x_5)$:
\begin{equation}
    \begin{split}
        p_1(\textbf{x}) := x_0+x_1+x_2+x_3+x_4+x_5 \\
        p_2(\textbf{x}) := x_0x_1+x_0x_2+\dots x_4x_5\\
        p_3(\textbf{x}) :=x_0x_1x_2+\dots x_3x_4x_5 
    \end{split}
\end{equation}
Recall the Newton identity:
\begin{equation}
    x_0^3+x_1^3+x_2^3+x_3^3+x_4^3+x_5^3=p_1(\textbf{x})^3-3p_1(\textbf{x})p_2(\textbf{x})+3p_3(\textbf{x})^3
\end{equation}

Set $\textbf{x}^d=(x_0^d,\dots,x_5^d)$, then:
\begin{equation}
    x_0^m+x_1^m+x_2^m+x_3^m+x_4^m+x_5^m=p_1(\textbf{x}^d)^3-3p_1(\textbf{x}^d)p_2(\textbf{x}^d)+3p_3(\textbf{x}^d)^3
\end{equation}

Let $W$ denotes the following variety in $\mathbb{P}^5$:
\begin{equation}
    p_1(\textbf{x}^d)=p_2(\textbf{x}^d)=p_3(\textbf{x}^d)=0
\end{equation}

By construction, $W\subseteq X^4_m$ is a subvariety of codimension 2, so $[W]\in Hdg^2(X^4_m)$ .
\begin{question}
Can $[W]$ project non trivially in $V(\alpha)$ for every $\alpha\in\mathfrak{B}^4_m$ which is not quasi-decomposable and not of standard type?
\end{question}
If the answer is yes, then we would have a positive answer to the Hodge conjecture in this case.

We know by Schur's lemma that the number of indecomposable elements is finite. Given $m\in\mathbb{Z}_+$, in order to prove the Hodge conjecture for $X^n_m$ and any $n$, it's enough to prove for all $X^{n}_m, n\leq n'$, where $n'=2(m'-1)$ and $m'$ is the largest length of all the indecomposable elements in $M_m$. 

Let $\mathcal{I}_m$ be set of indecomposable elements of $M_m$. Define $\phi : \mathbb{Z}_+\to \mathbb{Z}_+$ by the rule 
\begin{equation}
    \phi(m)=\{ \max{y} \:| \: (x_1,\dots x_{m-1},y)\in\mathcal{I}_m \}
\end{equation}
We have the following:
\begin{proposition}
If the Hodge conjecture is true for $X^{n}_m$, for all $n\leq 2(\phi(m)-1)$, then it's true for $X^n_m$ and any $n$.
\end{proposition}
\begin{proof}
The Hodge classes in $X^n_m$ are parametrized by $\mathcal{B}^n_m$, which can be viewed inside $M_m$ as elements of length $\frac{n}{2}+1$. Since the indecomposables generate $M_m$, it is enough that those be classes of algebraic cycles. But that is the case if the Hodge conjecture is true when $\frac{n}{2}+1\leq \phi(m)$, by definition of $\phi(m)$.
\end{proof}
Therefore, for Fermat varieties of degree $m$, we don't need to check the Hodge conjecture in every dimension. It's enough to prove the result for dimension up to $2(\phi(m)-1)$.

A natural question that arises is then what is the explicit expression of the function $\phi(m)$. For $m$ prime or $m=4$, we know already that $\phi(m)=1$. Also, by \cite{A}, we know that for $p>2$ prime $\phi(p^2)=\frac{p+1}{2}$. Here's a table with the a few values of $\phi(m)$:
\vspace{0.5cm}
\begin{center}
\begin{tabular}{ |c|c| } 
 \hline
 $m$ & $\phi(m)$ \\ 
 \hline
 20 & 5  \\ 
 21 & 3  \\ 
 22 & 7  \\ 
 23 & 1  \\ 
 24 & 9  \\ 
 25 & 3  \\ 
 \hline
\end{tabular}
\begin{tabular}{ |c|c| } 
 \hline
 $m$ & $\phi(m)$ \\ 
 \hline
 26 & 7  \\ 
 27 & 5  \\ 
 28 & 7  \\ 
 29 & 1  \\ 
 30 & 9  \\
 31 & 1  \\
 \hline
\end{tabular}
\begin{tabular}{ |c|c| } 
 \hline
 $m$ & $\phi(m)$ \\ 
 \hline
 32 & 9  \\ 
 33 & 5  \\ 
 34 & 5  \\ 
 35 & 8  \\ 
 36 & 13  \\
 37 & 1  \\
 \hline
\end{tabular}
\begin{tabular}{ |c|c| } 
 \hline
 $m$ & $\phi(m)$ \\ 
 \hline
 38 & 11  \\ 
 39 & 5  \\ 
 40 & 17  \\ 
 41 & 1  \\ 
 42 & 11  \\
 43 & 1  \\
 \hline
\end{tabular}
\end{center}
\vspace{0.5cm}
Based on the values above and the ones already computed, we believe the following is true:
\begin{conjecture}
For $p>2$ prime, we have $\phi(p^k)=\frac{p^{k-1}+1}{2}$, and $\phi(2^l)=2^{l-2}+1$ for $l>2$.
\end{conjecture}
Computing $\phi(m)$ for $m<48$ gives the following:
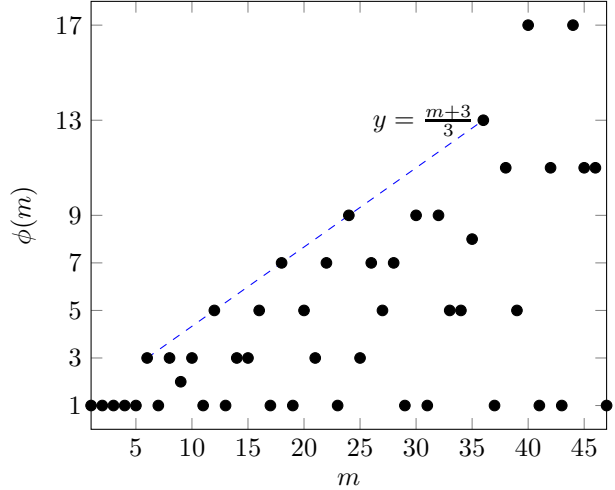
\vspace{.5cm}
\begin{center}
\begin{tikzpicture}
  \begin{axis}[xmin=1, xmax=47, ymin=0, ymax=18, xlabel=$m$, ylabel=$\phi(m)$,ytick={1,3,5,7,9,13,17},xtick={5,10,15,20,25,30,35,40,45}]
    \addplot[only marks] coordinates {
    (1,1)(2,1)(3,1)(4,1)(5,1)(6,3)(7,1)(8,3)(9,2)(10,3)(11,1)(12,5)(13,1)(14,3)(15,3)(16,5)(17,1)(18,7)(19,1)(20,5)(21,3)(22,7)(23,1)(24,9)(25,3)(26,7)(27,5)(28,7)(29,1)(30,9)(31,1)(32,9)(33,5)(34,5)(35,8)(36,13)(37,1)(38,11)(39,5)(40,17)(41,1)(42,11)(43,1)(44,17)(45,11)(46,11)(47,1)
    };
    \draw[blue,dashed](6,3)  -- (36,13)  node[left,black] {$y=\frac{m+3}{3}$};;
  \end{axis}
\end{tikzpicture}
\end{center}
\vspace{.5cm}
For $m\geq 48$, computations become more and more time consuming, and specially if $m$ has a lot of prime powers in its prime decomposition. But the results obtained here give us a glimpse about the structure of $M_m$ and consequently, the Hodge conjecture in the case of Fermat varieties.
\section*{Appendix: A code implementation of Shioda's approach}\label{appendix}
The following SAGE code, for a fixed $m>1$, checks whether or not every indecomposable element is quasi-decomposable. If such condition is true then the Hodge conjecture holds for $X^n_m$.
\begin{lstlisting}
from itertools import product
from fractions import gcd
import sys
import numpy as np
#from sage.geometry.polyhedron.backend_normaliz import Polyhedron_normaliz

def how_many_indec(m):
    p = MixedIntegerLinearProgram(base_ring=QQ)
    w = p.new_variable(integer=True, nonnegative=True)
    for k in range(1,m):
        if gcd(k,m) == 1:
            l=0
            for i in range(1,m):
                l += ((i*k) % m)*w[i-1]
            l += -m*w[m-1]
            #print l
            p.add_constraint(l == 0)
    p.add_constraint(w[m-1] >= 1)
    indec = p.polyhedron(backend='normaliz').integral_points_generators()[0]
    indec_less = [ x for x in indec if x[-1]> 2]
    return len(indec_less)

def poly_sol(m):
    p = MixedIntegerLinearProgram(base_ring=QQ)
    w = p.new_variable(integer=True, nonnegative=True)
    for k in range(1,m):
        if gcd(k,m) == 1:
            l=0
            for i in range(1,m):
                l += ((i*k) % m)*w[i-1]
            l += -m*w[m-1]
            p.add_constraint(l == 0)
    p.add_constraint(w[m-1] >= 1)
    return p.polyhedron(backend='normaliz')

def lengthOne(m):
    p = MixedIntegerLinearProgram(base_ring=QQ)
    w = p.new_variable(integer=True, nonnegative=True)
    for k in range(1,m):
        if gcd(k,m) == 1:
            l=0
            for i in range(1,m):
                l += ((i*k) % m)*w[i-1]
            l += -m*w[m-1]
            p.add_constraint(l == 0)
    p.add_constraint(w[m-1] == 1)
    return p.polyhedron(backend='normaliz').integral_points()

def get_indec(m):
    p = MixedIntegerLinearProgram(base_ring=QQ)
    w = p.new_variable(integer=True, nonnegative=True)
    #print 'x is %d and m is %d' % (x,m)
    for k in range(1,m):
        if gcd(k,m) == 1:
            l=0
            for i in range(1,m):
                l += ((i*k) % m)*w[i-1]
            l += -m*w[m-1]
            #print l
            p.add_constraint(l == 0)
    p.add_constraint(w[m-1] >= 1)
    return p.polyhedron(backend='normaliz').integral_points_generators()[0]


def get_standard(m,primes):
    result = []
    for p in primes:
        d = m/p
        if p == 2:
            for i in range(1,m):
                if (p*i) % m != 0 :#(d/gcd(i,d))>2:#(p*i) % m != 0 and 2*((p*i) % m) != m:##
                    temp = [i,(i+d) % m,(m-2*i) % m,d]
                    #print temp
                    std = []
                    for e in range(1,m):
                        std.append(temp.count(e))
                    std.append(2)
                    if tuple(std) not in result:
                        result.append(tuple(std))
        else:
            for i in range(1,m):
                if (p*i) % m != 0: #and 2*((p*i) % m) != m:#d/gcd(i,d)>2:
                    #print i
                    temp = [0]*(p+1)
                    for k in range(p):
                        temp[k]= (i+k*d) % m
                    temp[p]=(m-p*i) % m
                    #print temp
                    std = []
                    for e in range(1,m):
                        std.append(temp.count(e))
                    std.append((p+1)/2)
                    #print gcd(i,d)
                    #print tuple(std)
                    #print '--'
                    if p%2 == 1:
                        if tuple(std) not in result:
                            result.append(tuple(std))
                    else:
                        if tuple(std) not in result:
                            result.append(tuple(2*x for x in std))
    return result

def reverse_to(y,m):
    r=[]
    n = len(y)-2
    for e in range(1,m):
        r.append(y.count(e))
    r = r + [n/2 + 1]
    return tuple(r)

def convert_to_u(x,m):
    last = x[-1]
    n = 2*(last-1)
    r = []
    for k in range(m-1):
        if x[k] != 0:
            r = r + [k+1]*x[k]
    return tuple(r)

def get_points_length_less_m(x,m):
    p = MixedIntegerLinearProgram(base_ring=QQ)
    w = p.new_variable(integer=True, nonnegative=True)
    #print 'x is %d and m is %d' % (x,m)
    for k in range(1,m):
        if gcd(k,m) == 1:
            l=0
            for i in range(1,m):
                l += ((i*k) % m)*w[i-1]
            l += -m*w[m-1]
            #print l
            p.add_constraint(l == 0)
    p.add_constraint(w[m-1] >= 1)
    p.add_constraint(w[m-1] <= x)
    return p.polyhedron(backend='normaliz').integral_points()

arr = []

def get_indec_less(m,prm):
    p = poly_sol(m)
    print 'getting indecomposable elements for |m= %d| ...' % m
    indec = p.integral_points_generators()[0]
    #length_one = [ x for x in indec if x[-1]==1]
    #print 'there are %d length one' % len(length_one)
    standards = [list(x) for x in get_standard(m,prm)]
    print 'there are %d STANDARDS ELEMENTS' % len(standards)
    indec_less = [ x for x in indec if x[-1]>= 3 and list(x) not in standards]
    print 'there are %d indec of length>=3' % len(indec_less)
    return indec_less
    
def prime_factors(n):
    i = 2
    factors = []
    while i * i <= n:
        if n % i:
            i += 1
        else:
            n //= i
            if i not in factors:
                factors.append(i)
    if n > 1:
        if n not in factors:
            factors.append(n)
    #print('finished computing primes:',factors)
    return factors
m = 21
primes = prime_factors(m)
quasi = []
dict_ = {}
indec_less = get_indec_less(m,primes)
lasts_ =[]
length_one = lengthOne(m)
for el in indec_less:
    last = el[-1]
    print '---'
    print el
    print 'position: %d' % indec_less.index(el)
    count=0
    if last not in dict_:
        possible = get_points_length_less_m(last,m)
        dict_[last] = possible
        for el2,el3,el4 in product(length_one,possible,possible):
            if el + el2 == el3 + el4 and (el != el3 and el != el4):
                print 'I am quasi'
                quasi.append(el)
                break
            samples=len(possible)*len(possible)*len(length_one)
            count+=1
            sys.stdout.write("Progress: %.2f%%   \r" % (float(100*count)/samples))
            sys.stdout.flush()
        if el not in quasi:
            print 'This element is not quasi'
            print 'The HC CAN NOT be predicted for degree %d using this method, there are only %d quasi of %d' % (m,len(quasi),len(indec_less))
            break
    else:
        for el2,el3,el4 in product(length_one,dict_[last],dict_[last]):
            if el + el2 == el3 + el4 and (el != el3 and el != el4):
                print 'I am quasi'
                quasi.append(el)
                break
            samples=len(dict_[last])*len(dict_[last])*len(length_one)
            count+=1
            sys.stdout.write("Progress: %.2f%%   \r" % (float(100*count)/samples))
            sys.stdout.flush()
        if el not in quasi:
            print 'This element is not quasi'
            print 'The HC CAN NOT be predicted for degree %d using this method, there are only %d quasi of %d' % (m,len(quasi),len(indec_less))
            break
print 'The HC is TRUE for degree %d fermats' % m
\end{lstlisting}
\bibliographystyle{amsplain}

\end{document}